\documentclass[12pt]{article}

\usepackage[english]{babel}
\usepackage[utf8]{inputenc}
\usepackage{amsmath}
\usepackage{amssymb}
\usepackage{amsthm}
\usepackage{graphicx}
\usepackage[colorinlistoftodos]{todonotes}
\usepackage[export]{adjustbox}
\usepackage{subcaption}
\usepackage{float}
\usepackage{array, booktabs}
\usepackage{mathtools}

\DeclareUnicodeCharacter{00A0}{ }
\addto{\captionsenglish}{}
\usepackage{natbib}

\newtheorem{lemma}{Lemma}
\newtheorem{theorem}{Theorem}
\newtheorem{definition}{Definition}

\title{On Asymptotic Properties of the Separating Hill Estimator}
\author{Matias Heikkilä\footnote{Aalto University School of Science} \and Yves Dominicy\footnote{Université libre de Bruxelles --- Yves Dominicy acknowledges financial support via a Mandat de Chargé de Recherces FNRS} \and Pauliina Ilmonen\footnote{Aalto University School of Science}}

\begin{document}

\maketitle

\begin{abstract}
Modeling and understanding multivariate extreme events is challenging, but of great importance in various applications --- e.g. in biostatistics, climatology, and finance. The separating Hill estimator can be used in estimating the extreme value index of a heavy tailed multivariate elliptical distribution. We consider the asymptotic behavior of the separating Hill estimator under estimated location and scatter. The asymptotic properties of the separating Hill estimator are known under elliptical distribution with known location and scatter. However, the effect of estimation of the location and scatter has previously been examined only in a simulation study. We show, analytically, that the separating Hill estimator is consistent and asymptotically normal under estimated location and scatter, when certain mild conditions are met.
\end{abstract}

\section{Introduction}

Extreme value theory is used in understanding unlikely observations. It provides a theoretical framework that classifies distributions according to their tail behavior and explains asymptotic properties of the extreme order statistics. It also provides statistical estimators for different interesting quantities. There is a range of applications including finance, insurance, climatology and geology.

The possibility to handle multivariate observations is vital for many applications. The traditional way to approach the multivariate setting is to consider the componentwise maxima of the multivariate observations and apply the theory of the one dimensional case to the marginal distributions (see \cite{deHaan} chapter 6). \citet{Ilmonen} approach the problem in a rather different way by considering a situation where the symmetry properties of the distribution provide a natural order relation and hence a natural concept of a multivariate extreme order statistic. Namely, they consider the extreme values of elliptically distributed random variables.

Extreme value index encodes information about the tail behavior of univariate distributions. In the case of a positive extreme value index, a popular estimator for this parameter is the Hill estimator proposed in \cite{Hill}. A multivariate generalization of this, the separating Hill estimator, was recently introduced in \cite{Ilmonen}. Below we review the Hill estimator and its asymptotic properties. We also discuss the definition of a positive extreme value index of an elliptical distribution. In the remainder we settle a severe deficiency in the current asymptotic theory of the separating Hill estimator.

The value of the separating Hill estimator depends on variables that describe the location and scatter of the underlying distribution. We show that, under mild conditions, the asymptotic properties of the estimator remain unchanged if, instead of true location and scatter, the estimator is evaluated with respect to estimated location and scatter. This is vital for applications, since in practice estimated location and scatter are what one has at one's disposal.

This paper is organized as follows. In Section \ref{sec::positive} we review the relevant basics of the Hill estimator, in Section \ref{sec::known} we review the separating Hill estimator and in Section \ref{sec::estimated} we obtain its asymptotic properties under estimated location and scatter.

\section{Estimation of a positive extreme value index}\label{sec::positive}

The extreme value index yields the limiting distribution of the extreme order statistics subject to a sequence of appropriately chosen affine normalizations. For example, extreme quantile estimation is based on having an estimate of this parameter. The extreme value index naturally arises from the Fisher-Tippett-Gnedenko theorem (see \citet{fisher} and \citet{gnedenko}) which is of fundamental nature for extreme value theory.
\begin{theorem}
Let $X_n$ be a sequence of i.i.d. real-valued random variables whose cumulative distribution function is $F$. If there are sequences of real numbers $a_n > 0$, $b_n$ such that
\begin{equation*}
P\left(   \frac{ \max \left\{ X_1, \dots, X_n \right\} - b_n}{a_n}  \leq x \right) = F^n \left( a_n x + b_n \right) \to G(x)
\end{equation*}
as $n \to \infty$ for all $x\in \mathbb{R}$, then there is $\gamma \in \mathbb{R}$ such that
\begin{equation*}
G(x)=G_\gamma (x) = \exp \left( - \left(1+ \gamma x \right)^{- 1 / \gamma }\right)
\end{equation*}
for $1+ \gamma x \geq 0$ and $0$ otherwise. In the case $\gamma =0$
\begin{equation*}
G_0 (x) = \exp\left( - e^{-x} \right).
\end{equation*}
\end{theorem}
\noindent
The following definitions are now natural.
\begin{definition}
Let $F$ be a cumulative distribution function such that there are sequences of real numbers $a_n > 0$, $b_n$ satisfying
\begin{equation*}
F^n \left( a_n x + b_n \right) \to G(x) = G_\gamma (x),
\end{equation*}
as $n \to \infty$. Then $F$ is said to be in the domain of attraction of $G_\gamma$, denoted $F \in \mathcal{D}\left( G_\gamma \right)$.
\end{definition}

\begin{definition}
The extreme value index of a distribution function $F$ is $\gamma$ if and only if $F \in \mathcal{D}\left(  G_\gamma \right)$.
\end{definition}

\bigskip

There are several statistical estimators for $\gamma$ which are valid in different cases and have different desirable properties. Among the earliest ones is the Hill estimator proposed in \citet{Hill}. It is valid for $\gamma > 0$ and among its merits is its straightforward applicability. Since it was proposed, the Hill estimator has indeed become a widely used tool.

Consider a sample $s_1, s_2, \dots, s_n \in \mathbb{R}$. Denote the order statistics by $s_{(1,n)} \geq s_{(2,n)} \geq \cdots \geq s_{(n,n)}$. Let $1 \leq k < n$. The expression for the Hill estimator is
\begin{equation*}
\hat{H}_{k,n} = \frac{1}{k} \sum \limits_{i=1}^k \log \left( \frac{s_{(i,n)}}{s_{(k+1,n)}} \right).
\end{equation*}
The restriction $\gamma > 0$ is clear since each of the logarithms takes a positive value.

\bigskip

There are characterizations for the distributions in each $\mathcal{D} \left( G_\gamma \right)$. For $\gamma >0$ the characteristic property is a regularly varying tail. The general definition of this property is as follows.
\begin{definition}
Let $f: \mathbb{R}_+ \to \mathbb{R}$ be an eventually positive function. If for some $\alpha \in \mathbb{R}$
\begin{equation*}
\lim \limits_{t\to \infty}\frac{f(tx)}{f(t)} = x^{\alpha}
\end{equation*}
holds for all $x >0$, $f$ is said to be regularly varying, denoted $f \in RV_\alpha$.
\end{definition}
\noindent
The number $\alpha$ is called the index of regular variation of $f$. If $\alpha = 0$, $f$ is said to be slowly varying. Regularly varying functions can be considered as a generalization of functions of the form $f(x) = Cx^{\alpha}$, $C \in \mathbb{R},$ to the case where the constant $C$ is replaced with a slowly varying function.

The distributions in $\mathcal{D}\left( G_\gamma \right)$ for $\gamma >0$ are precisely those whose tails are regularly varying.
\begin{theorem}\label{thm::RVandDomain}
Let $F$ be a distribution function. The condition $F \in \mathcal{D} (G_\gamma)$ for $\gamma >0$ holds if and only if
\begin{equation*}
\lim \limits_{t \to \infty} \frac{1-F(tx)}{1-F(t)} = x^{- 1/\gamma}.
\end{equation*}
\end{theorem}
\noindent
In other words: $F \in \mathcal{D} \left(G_\gamma \right)$ for $\gamma > 0$ if and only if $1-F \in RV_{-\alpha}$, where $\alpha = 1 / \gamma$. In this context $\alpha = 1 / \gamma$ is called the tail index of the distribution.

\bigskip

Asymptotic properties of the Hill estimator have been a subject to much interest. Consider the case where $F$ is a Pareto distribution i.e. $F(x) =1- C x^{-\alpha}$, $x \geq x_m(C) >0$, for some $\alpha>0$. Now the self-similarity properties of the distribution suggest that, intuitively, $\hat{H}_{k,n} \to \gamma$ whenever $k=k(n)$ grows large as $n \to \infty$. This indeed turns out to be the case and the result even holds more generally: When $C$ is replaced by a slowly varying function, i.e. when $1-F$ is regularly varying.

\begin{theorem}\label{thm::HillConsistency}
Assume that $F \in \mathcal{D}(G_\gamma)$ for $\gamma > 0$. Now, as $k_n \to \infty,$ $n \to \infty,$ $k_n /n \to 0$,
\begin{equation*}
\hat{H}_{k_n,n} \to_P \gamma.
\end{equation*}
\end{theorem}
\noindent
This result was obtained in \citet{Mason}.

The natural question after settling the consistency of $\hat{H}_{k,n}$ is its limiting distribution. This requires some additional assumptions about the underlying distribution. The following class of functions that is appropriate for this purpose is discussed in \citet{2erv}.

\begin{definition}\label{def::2ERV}
A measurable function $f: \mathbb{R}_+ \to \mathbb{R}$ is said to be of second order extended regular variation  if \begin{equation*}
\lim \limits_{t \to \infty} \frac{\frac{f(tx)-f(t)}{a(t)} -\frac{x^{\gamma}-1}{\gamma}}{A \left( t \right) } = H_{\gamma, \rho}(x) = c_1 \int_1^x s^{\gamma-1} \int_1^s u^{\rho-1 } \mathrm{d}u\mathrm{d}s + c_2\int_1^x s^{\gamma+\rho-1} \mathrm{d}s
\end{equation*}
for some $c_1,c_2 \in \mathbb{R}$, $\gamma \in \mathbb{R}$ and $\rho \leq 0$, where $H$ is not a multiple of $\frac{x^{\gamma}-1}{\gamma}$ and the positive or negative function $A$ converges to zero as $t \to \infty$. The function $a$ is some positive auxiliary function. We denote this by $f \in 2ERV_{\gamma, \rho}$.
\end{definition}

\noindent
Let $F$ be a distribution function. Consider the following function
\begin{equation}\label{eq::U}
U(y) = \inf \left\{ x \in \mathbb{R} \: \middle| \frac{1}{1-F(x)} \leq y  \: \right\}.
\end{equation}
It can be shown that $U \in 2ERV_{\gamma, \rho}$ for $\gamma >0$ implies $F \in \mathcal{D}( G_\gamma)$. The asymptotic behavior of $\hat{H}_{k,n}$ is also neatly expressed in terms of $U$.
\begin{theorem}\label{thm::univariateHillDistribution}
Let $F$ be a distribution function such that the related $U$, given by \eqref{eq::U}, is in $2ERV_{\gamma, \rho}$ for some $\gamma >0$. Let $A$ be the auxiliary function of $U$ in Definition \ref{def::2ERV}. Now, as $k_n \to \infty,$ $n \to \infty,$ $k_n /n \to 0$,
\begin{equation*}
\sqrt{k_n} \left( \hat{H}_{k_n,n} - \gamma \right) \to \mathcal{N} \left( \frac{\lambda}{1-\rho}, \gamma^2 \right),
\end{equation*}
where
\begin{equation*}
\lambda = \lim \limits_{n \to \infty} \sqrt{k_n} A\left( \frac{n}{k_n} \right).
\end{equation*}
\end{theorem}
\noindent
It is worth a remark that if
\begin{equation*}
\lim \limits_{t\to \infty} t^\alpha \left( \frac{U(tx)}{U(t)} - x^\gamma \right)=0
\end{equation*}
for all $\alpha >0$, the result above holds with $\lambda=0$, i.e. the estimator is then asymptotically unbiased.

The general conditions for the tail of the distribution under which the sequence $k_n$ can be selected so that the Hill estimator is asymptotically normal were derived in  \citet{haeusler}. The roles of different smoothness conditions for the asymptotic normality of the Hill estimator were further clarified in \citet{deHaanHillNormal}.

Together these results provide a satisfying picture of the basic asymptotic properties of the Hill estimator. They were also the foundation for the asymptotic properties of the separating Hill estimator derived in \citet{Ilmonen}. Other aspects of the asymptotic behavior that have been examined include asymptotic bias of the estimator (see e.g. \cite{haanPeng} and \citet{haeusler}), optimal selection of the sequence $k_n$ (see \citet{hall2}, \citet{gomes} and \citet{danielsson}) and bias correction (see \citet{gomes2} and references therein).

\section{Separating Hill estimator under known location and scatter}\label{sec::known}

In the univariate context, extreme values are observations that are exceptionally large or small. In the multivariate setting, the situation is less straightforward due to the lack of a canonical order relation. A possible approach explored in e.g. \citet{Sibuya}, \citet{Tiago} and \citet{deHaanResnickMulti} is to consider the componentwise maxima of the sample. However, it is questionable whether this approach gives a good definition of multivariate extremes. It is not invariant under affine transformations. In fact, a simple rotation may alter the data points that are considered as extreme observations. Under the assumption of multivariate ellipticity, the symmetry properties allow for a different approach.

A random variable is said to be elliptically distributed if
\begin{equation}\label{eq::elliptical}
X \sim \mu + \mathcal{R} \Lambda U,
\end{equation}
where the location vector $\mu \in \mathbb{R}^d$, the random vector $U$ is uniformly distributed on the unit sphere $\mathbb{S}^{d-1} \subset \mathbb{R}^d$, the matrix $\Lambda \in \mathbb{R}^{d \times d}$ is such that the scatter matrix $\Sigma = \Lambda \Lambda^T \in \mathbb{R}^{d \times d}$ is a full rank positive definite matrix, and $\mathcal{R}$ is a real valued random variable. The random variable $\mathcal{R}$ is called the generating variate of the distribution.

The family of elliptical distributions was introduced in \citet{kelker}. (See also \citet{fang}.)

The symmetry properties of elliptical distributions motivate considering the Mahalanobis distance introduced in \citet{mahalanobis}. (See also \citet{mahalanobis2}.)
\begin{definition}
Let $A \in \mathbb{R}^{d\times d}$ be a full rank symmetric positive definite matrix. The metric $d_A$ given by
\begin{equation*}
d_A(x,y) = \left( \left\langle \, x-y \, \middle| \, A^{-1} \left( x-y\right)\, \right\rangle \right)^{\frac{1}{2}}
\end{equation*}
is called the Mahalanobis distance relative to $A$.
\end{definition}

\noindent
The crucial observation here is the following: If $X$ follows an elliptical distribution with the location $\mu$ and scatter $\Sigma,$ the Mahalanobis distance $d_\Sigma(X, \mu)$ is equal in distribution to $\mathcal{R}$ as
\begin{eqnarray*}
d_\Sigma(X, \mu)^2 &\sim & \left\langle \, \mu+ \mathcal{R}  \Lambda U -\mu \, \middle| \Sigma^{-1} \left(\mu + \mathcal{R}  \Lambda U -\mu \right) \, \, \right\rangle  \\ &\sim & \mathcal{R}^2 \left\langle \,    U \, \middle| \Lambda^T \left( \Lambda \Lambda^T \right)^{-1}  \Lambda U \, \, \right\rangle  \\ &\sim& \mathcal{R}^2   \left\| U \right\|^2 \sim \mathcal{R}^2.
\end{eqnarray*}
Despite its simplicity, this observation rigorously describes the intimate relationship between $X$ and $\mathcal{R}$. It leads us to consider extreme observations under ellipticity to be the ones that correspond to extreme values of $\mathcal{R}$ in the univariate sense. Consequently, under ellipticity, we have both intuitively and formally sensible concept of an extreme value index.

\begin{definition}\label{def::ellipticalExtreme}
An elliptical distribution is said to have an extreme value index $\gamma >0$ if and only if the extreme value index of its generating variate is $\gamma >0$.
\end{definition}

The definition above is further supported by a result from \citet{ellipticalRV}: General approach to the multivariate regular variation is defining it in the following way.
The distribution of a $d$-dimensional random variable $X$ is said to be regularly varying with index $\alpha>0$ if there exists a random variable $\Theta$ with values on the unit sphere $\mathbb{S}^{d-1}$ a.s. such that, for all $x>0$, as $t \to \infty$,
\begin{equation*}
\frac{P \left( \left\| X \right\| \geq tx  , \, X / \left\| X \right\| \in \cdot \right)}{P \left( \left\| X \right\| \geq t  \right)} \to_v x^{-\alpha} P \left( \Theta \in \cdot \right),
\end{equation*}
where $\to_v$ denotes vague convergence. It was shown in \citet{ellipticalRV}, that the regular variation of an elliptical distribution (defined in the general sense) is equivalent to the regular variation of the tail of the corresponding generating variate. Recall now that by Theorem \ref{thm::RVandDomain}, a positive extreme value index is equivalent to having regularly varying tail.

\bigskip

An affine invariant multivariate extension of the Hill estimator was introduced in \citet{Ilmonen}: Consider a multivariate sample $S_n=\left\{X_1, \dots, X_n \right\} \subset \mathbb{R}^d$. Let $\mu \in \mathbb{R}^d$ and let $\Sigma \in \mathbb{R}^{d\times d}$ be a full rank positive definite matrix. Let $D_{(1,n)} \geq D_{(2,n)} \geq \cdots \geq D_{(n,n)}$ be the order statistics of the Mahalanobis distances $d_\Sigma(X_i,\mu)$. The separating Hill estimator under these parameters is given by
\begin{equation}\label{eq::sepHill}
\hat{H}_d \left( S_n, \mu, \Sigma, k,n \right) = \frac{1}{k} \sum \limits_{i=1}^k \log \left( \frac{D_{(i,n)}}{D_{(k+1,n)}} \right).
\end{equation}

Under ellipticity, as observed above, if the location and scatter of the underlying distribution are $\mu$ and $\Sigma$ respectively, the Mahalanobis distance $d_\Sigma(X,\mu)$ is equal in distribution to $\mathcal{R}$. In that case the asymptotic properties of $\hat{H}^s_d$ under known location and scatter are straightforward to derive. This was done in \citet{Ilmonen}. However, in practice, one has to
 estimate the location and scatter. The asymptotic theory under known location and scatter is thus highly insufficient for ensuring the practical applicability of the estimator.

In this paper we settle the matter by proving that replacing location and scatter by estimates does not affect the asymptotic behavior of the separating Hill estimator. This holds as long as the estimators for the location and scatter are $\sqrt{n}$-consistent.

\section{Separating Hill estimator under estimated location and scatter}\label{sec::estimated}

It is not trivial that the asymptotic properties of the separating Hill estimator  are not affected by replacing the true location and scatter in the expression
\eqref{eq::sepHill} with estimates.

Let $d \geq 2$. Throughout this section, including Subsections \ref{sec::Lemmas} and \ref{sec::main}, we assume that $X_1, X_2, \dots,$  $X_n: \Omega \to \mathbb{R}^d$, is a sequence of i.i.d. random variables with a $d$-dimensional elliptical distribution \eqref{eq::elliptical}. We assume that the distribution has a positive extreme value index $\gamma$. We denote the location parameter of the distribution by $\mu$, and the positive definite scatter parameter by $\Sigma$.
Notations $\hat{\mu}_n$ and $\hat{\Sigma}_n$ are used for estimators of the location and scatter of the distribution, respectively.

Throughout this section, again including Subsections \ref{sec::Lemmas} and \ref{sec::main}, we write
\begin{eqnarray*}
R_i &=& \left\langle \, X_i - \mu  \, \middle| \, \Sigma^{-1} \left( X_i - \mu \right) \, \right\rangle^{1/2} \quad \\ \quad E_i^{(n)}  &=& \left\langle \, X_i - \hat{\mu}_n  \, \middle| \, \hat{\Sigma}^{-1}_n \left( X_i - \hat{\mu}_n \right) \, \right\rangle^{1/2}.
\end{eqnarray*}
That is, $R_i$ is the true Mahalanobis distance of $X_i$ from the mean and $E_i^{(n)}$ is the $n$:th estimate of it. We denote their order statistics by $R_{(1,n)} \geq R_{(2,n)} \geq \dots \geq R_{(n,n)}$ and $E_{(1,n)} \geq E_{(2,n)} \geq \dots \geq E_{(n,n)}$,
and by $S_n$ we denote the set of the $n$ first observations $X_i$, $S_n = \left\{ X_1 ,\dots, X_n \right\}$. 

Below we will show that the difference
\begin{equation*}
\left| \hat{H}_d(S_n,\mu,\Sigma,k_n,n) - \hat{H}_d(S_n,\hat{\mu}_n,\hat{\Sigma}_n,k_n,n) \right|
\end{equation*}
becomes negligible if $k_n \to \infty,$ $n \to \infty,$ $k_n / n \to 0$ and if the estimates $\hat{\mu}_n$ and $\hat{\Sigma}_n$ behave well.

Our strategy is straightforward. The difference above consists of log-ratios
\begin{equation}\label{eq::logratio}
\log \left( \frac{E_{(i,n)}}{R_{(i,n)}} \right),
\end{equation}
where $ 1 \leq i \leq k_n+1$. We use an elementary argument to bound the absolute value of these expressions uniformly by a sequence that approaches zero in probability sufficiently quickly. The desired results are then obtained by showing that the obtained bound holds with a probability that approaches one.

\subsection{Controlling the log-ratios of the order statistics}\label{sec::Lemmas}
The asymptotical results for the separating Hill estimator under estimated location and scatter are obtained in two parts. We begin by proving the following lemmas that give a bound for the individual log-ratios \eqref{eq::logratio}. We will prove that the bound is valid for the pairs $\left( R_{(i,n)},E_{(i,n)} \right)$, where $1 \leq i \leq l$, if $l$ satisfies certain condition and $n$ is large enough.

\vspace{12pt}

\begin{lemma}\label{Lemma::epsilons}
Let $X_1, X_2, \dots$ be a sequence of i.i.d. random variables with a $d$-dimensional elliptical distribution \eqref{eq::elliptical}.  Let $R_{(i,n)}$ and $E_{(i,n)}$ be defined as at the beginning of Section \ref{sec::estimated}. For all $n > 0$ and $ 1 \leq i \leq n,$ define a random variable $\varepsilon_{(i,n)}$ by $\varepsilon_{(i,n)}= E_{(i,n)}^2 - R_{(i,n)}^2.$

Let $1 \leq l \leq n$. If $M_n < 1$ and $R_{(l,n)} > \frac{M_n}{2 \left( 1- M_n \right)},$ then
\begin{equation*}
\left| \varepsilon_{(l, n)} \right| \leq  \delta_n \left( R_{(l,n)}  \right),
\end{equation*}
where the polynomial
\begin{equation*}
\delta_n \left( x  \right) =  M_n x^2 + M_n x  + M_n,
\end{equation*}
and the coefficient $M_n$ in the above expression is given by
\begin{eqnarray}\label{eq::Mn}
M_n &=& \max \bigg\{ \lambda_{\max} \, A_n, \, \sqrt{\lambda_{\max}} \left(2 \left\| \mu \right\| A_n + B_n \right) , \\&& \, A_n \left\| \mu \right\|^2 + B_n \left\| \mu \right\| + C_n \, \bigg\}, \nonumber
\end{eqnarray}
where $\lambda_{\max}$ denotes the largest eigenvalue of $\Sigma$ and
\begin{eqnarray*}
A_n&=& \left\| \Sigma^{-1} - \hat{\Sigma}^{-1} \right\| \\
B_n &=& \left( \left\| \hat{\mu}_n \right\| + \left\| \mu \right\| \right)\left\| \Sigma^{-1} - \hat{\Sigma}^{-1}_n \right\| +    \left( \left\| \hat{\Sigma}^{-1}_n \right\| +  \left\| \Sigma^{-1} \right\| \right) \left\| \mu - \hat{\mu}_n \right\| \\
C_n &=& \left\| \mu \right\|^2 \left\| \Sigma^{-1} - \hat{\Sigma}^{-1}_n \right\| +  \left( \left\| \mu \right\|  + \left\| \hat{\mu}_n \right\| \right) \left\| \hat{ \Sigma}^{-1}_n \right\| \left\| \mu- \hat{\mu}_n \right\|.
\end{eqnarray*}
\end{lemma}
\begin{proof}
The scatter matrix $\Sigma$ is positive definite. Thus its eigenvalues $\lambda_1, \dots, \lambda_d$ are positive and a properly normalized set of its eigenvectors $\left\{ e_1 , \dots, e_d \right\}$ form an orthonormal basis for $\mathbb{R}^d$. Write
\begin{equation*}
\lambda_{\max} = \max  \left\{  \lambda_1, \dots, \lambda_n \right\}.
\end{equation*}
The following inequality holds for all $x,y \in \mathbb{R}^d$:
\begin{equation}\label{eq::mahalanobisBound}
\left\| x-y \right\| \leq \sqrt{\lambda_{\max}} \, d_\Sigma (x,y)
\end{equation}
This can be seen by writing $x$ and $y$ in the basis $\left\{ e_1, \dots, e_d \right\}$
\begin{eqnarray*}
\left\|  x-y \right\|^2 &=& \sum \limits_{i=1}^d \left( x_i - y_i \right)^2 \\ &=&  \lambda_{\max} \, \sum \limits_{i=1}^d  \frac{1}{\lambda_{\max}}  \left( x_i - y_i \right)^2 \\
&\leq & \lambda_{\max} \, \sum \limits_{i=1}^d  \frac{1}{\lambda_{i}}  \left( x_i - y_i \right)^2 \left\langle \, e_i \, \middle|\,  e_i \, \right\rangle \\
&=& \lambda_{\max} \left\langle \, \sum \limits_{j=1}^d \left( x_j - y_j \right)e_j \, \middle|\, \sum \limits_{i=1}^d   \left( x_i - y_i \right) \Sigma^{-1} e_i \, \right\rangle \\
&=& \lambda_{\max} \, d_\Sigma(x,y)^2.
\end{eqnarray*}

By the Cauchy-Schwartz inequality and the bound $\left\| A x\right\| \leq \left\| A \right\| \left\| x \right\|$ for the operator norm, we have  that for all $1 \leq i \leq n$:
\begin{eqnarray*}
\left| R_i^2 - \left( E_i^{(n)} \right)^2  \right| &=& \bigg|\left\langle \, X_i - \mu  \, \middle| \, \Sigma^{-1} \left( X_i - \mu \right) \, \right\rangle \\&& - \left\langle \, X_i - \hat{\mu}_n  \, \middle| \, \hat{\Sigma}^{-1}_n \left( X_i - \hat{\mu}_n \right) \, \right\rangle\bigg| \\
&\leq& \left\| \Sigma^{-1} - \hat{\Sigma}^{-1}_n \right\| \left\| X_i \right\|^2 + \biggl( \left( \left\| \hat{\mu}_n \right\| + \left\| \mu \right\| \right) \left\| \Sigma^{-1} - \hat{\Sigma}^{-1}_n \right\| \\&&  +    \left( \left\| \hat{\Sigma}^{-1}_n \right\| +  \left\| \Sigma^{-1} \right\| \right) \left\| \mu - \hat{\mu}_n \right\| \biggr) \left\| X_i \right\| \\&& + \biggl( \left\| \mu \right\|^2 \left\| \Sigma^{-1} - \hat{\Sigma}^{-1}_n \right\| \\&& +  \left( \left\| \mu \right\|  + \left\| \hat{\mu}_n \right\| \right) \left\| \hat{ \Sigma}^{-1}_n \right\| \left\| \mu- \hat{\mu}_n \right\|  \biggr).
\end{eqnarray*}
Denote the coefficients of $\left\| X_i \right\|^2$ and $\left\| X_i \right\|$ by $A_n$ and $B_n$, respectively, and denote the expression inside the last large brackets by $C_n$. By the estimate $\left\| X \right\| \leq \left\| X - \mu \right\| + \left\| \mu \right\|$ the expression above has an upper bound
\begin{equation*}
A_n \left\| X_i - \mu \right\|^2 + \left( 2 \left\| \mu \right\| A_n + B_n \right) \left\| X_i - \mu \right\| + \left( A_n \left\| \mu \right\|^2 + B_n \left\| \mu \right\| + C_n \right).
\end{equation*}
We can now relate $\left\| X_i - \mu \right\|$ and $\left\| R_i\right\|$ using Equation \eqref{eq::mahalanobisBound} and obtain the following upper bound for the original expression:
\begin{equation*}
\lambda_{\max} \, A_n R_i^2 + \sqrt{ \lambda_{\max} } \, \left( 2 \left\| \mu \right\| A_n + B_n \right) R_i + \left( A_n \left\| \mu \right\|^2 + B_n \left\| \mu \right\| + C_n \right).
\end{equation*}
An upper bound for this is $\delta_n \left( R_i \right) =M_n R_i^2 + M_n R_i +M_n$, where $M_n$ is as in \eqref{eq::Mn}.

Consider
\begin{equation*}
\varepsilon_{(i,n)} = E^{2}_{(i,n)} - R_{(i,n)}^2.
\end{equation*}
We bound $\left| \varepsilon_{(i,n)} \right|$ by finding an upper bound and a lower bound for the difference above. We also find the conditions under which these bounds are valid.

The upper bound: Let $1 \leq l \leq n$. Since $M_n > 0,$ the function
\begin{equation*}
x^2 + \delta_n(x) = (1+M_n) x^2 + M_n x + M_n
\end{equation*}
is strictly increasing for $x > 0$. Thus, for all $R_{j} \leq R_{(l,n)},$ we have that
\begin{equation*}
\left( E_{j}^{(n)} \right)^2 \leq R_j^2 + \delta_n \left( R_j \right) \leq R_{(l,n)}^2 + \delta_n \left(  R_{(l,n)}\right).
\end{equation*}
It now follows that $R_{(l,n)}^2 + \delta_n \left(  R_{(l,n)}\right)$ is an upper bound for at least $n-l+1$ of the observations $\left(E_{j}^{(n)} \right)^2$, or
\begin{equation*}
E_{(l,n)}^2 \leq R_{(l,n)}^2 + \delta_n \left(  R_{(l,n)}\right).
\end{equation*}
Equivalently
\begin{equation*}
\varepsilon_{(i,n)} \leq \delta_n \left(  R_{(l,n)}\right).
\end{equation*}

The lower bound: By differentiating we see that the function
\begin{equation*}
x^2 - \delta_n \left( x \right) = (1-M_n) x^2 -M_n x - M_n
\end{equation*}
is strictly increasing if
\begin{equation*}
\frac{2 \left( 1-M_n \right)}{M_n} \, x > 1.
\end{equation*}
Thus, assuming that $M_n <1$ and $R_{(l,n)} > \frac{M_n}{2 \left( 1-M_n \right)},$ we have, for all $R_i \geq R_{(l,n)},$ that
\begin{equation*}
\left( E_{j}^{(n)} \right)^2 \geq R_j^2 - \delta_n \left( R_j \right) \geq R_{(l,n)}^2 - \delta_n \left(  R_{(l,n)}\right).
\end{equation*}
Thus, under the given conditions, $R_{(l,n)}^2 - \delta_n \left(  R_{(l,n)}\right)$ is a lower bound for at least $l$ of the observations $\left(E_{j}^{(n)} \right)^2$, or
\begin{equation*}
E_{(l,n)}^2 \geq R_{(l,n)}^2 - \delta_n \left(  R_{(l,n)}\right).
\end{equation*}
Equivalently
\begin{equation*}
\varepsilon_{(i,n)} \geq - \delta_n \left(  R_{(l,n)}\right).
\end{equation*}
\end{proof}

\vspace{12pt}

Next we obtain a uniform bound for the large log-ratios in terms of $M_n$.

\vspace{12pt}

\begin{lemma}\label{Lemma::differences}
Let $X_1, X_2, \dots$ be a sequence of i.i.d. random variables with a $d$-dimensional elliptical distribution \eqref{eq::elliptical}. Let $R_{(i,n)}$ and $E_{(i,n)}$ be defined as at the beginning of Section \ref{sec::estimated} and let $M_n$ be defined as in Lemma \ref{Lemma::epsilons}.
Let $1 \leq l \leq n$. Assume that $M_n < 1$ and $R_{(l ,n )} \geq \frac{M_n}{2\left(1 - M_n \right)}$. Denote
\begin{equation*}
a_n = M_n + \frac{M_n}{R_{(l,n)}} + \frac{M_n}{R^2_{(l,n)}}.
\end{equation*}
If $a_n < 1$, then the following holds for all $1 \leq i \leq l$:
\begin{equation*}
\left| \log \left( \frac{E_{(i,n)}}{E_{(l,n)}} \right) - \log \left( \frac{R_{(i,n)}}{R_{(l,n)}} \right) \right| \leq \log \left( \frac{1}{1-a_n} \right).
\end{equation*}
\end{lemma}

\begin{proof}
By the triangle inequality
\begin{eqnarray*}
\left| \log \left( \frac{E_{(i,n)}}{E_{(l,n)}} \right) - \log \left( \frac{R_{(i,n)}}{R_{(l,n)}} \right) \right| &=& \frac{1}{2} \left| \log \left( \frac{E^2_{(i,n)}}{R^2_{(i,n)}} \right) + \log \left( \frac{R^2_{(l,n)}}{E^2_{(l,n)}} \right) \right| \\
&\leq& \frac{1}{2} \left| \log \left( \frac{E^2_{(i,n)}}{R^2_{(i,n)}} \right) \right| + \frac{1}{2} \left| \log \left( \frac{E^2_{(l,n)}}{R^2_{(l,n)}} \right) \right| \\ &=& \frac{1}{2} \left| \log \left( 1+ \frac{ \varepsilon_{(i,n)}}{R^2_{(i,n)}} \right) \right| \\&&+\frac{1}{2} \left| \log \left( 1+ \frac{ \varepsilon_{(l,n)}}{R^2_{(l,n)}} \right) \right|,
\end{eqnarray*}
where $\varepsilon_{(i,n)}$ are as in Lemma \ref{Lemma::epsilons}. Under the assumptions $M_n < 1$ and $R_{(l,n)} > \frac{M_n}{2 \left( 1-M_n \right)},$ Lemma \ref{Lemma::epsilons} yields
\begin{equation*}
\left| \varepsilon_{(i,n)} \right| \leq \delta_n \left(R_{(i,n)} \right)
\end{equation*}
for all $1 \leq i \leq l$. Since $R_{(i,n)} \geq R_{(l,n)}$, the quantity
\begin{equation*}
\frac{ \delta_n \left( R_{(i,n)} \right) }{R^2_{(i,n)}} = M_n + \frac{M_n}{R_{(i,n)}} +\frac{M_n}{R^2_{(i,n)}}
\end{equation*}
is bounded from above by
\begin{equation*}
a_n=M_n + \frac{M_n}{R_{(l,n)}} +\frac{M_n}{R^2_{(l,n)}}.
\end{equation*}

Assuming that $a_n <1$, by monotonicity of the logarithm, the following holds for all $1 \leq i \leq l$:
\begin{equation*}
\left| \log \left( 1+ \frac{ \varepsilon_{(i,n)}}{R^2_{(i,n)}} \right) \right| \leq \max  \left\{ \left| \log \left( 1 \pm \frac{ \delta_n \left( R_{(i,n)} \right) }{R^2_{(i,n)}} \right) \right| \right\} \leq \max \left\{ \left| \log \left( 1 \pm a_n \right) \right| \right\}.
\end{equation*}
The mean value theorem yields
\begin{equation*}
 \max \left\{ \left| \log \left( 1 \pm a_n \right) \right|\right\} = \left| \log \left( 1-a_n \right) \right| = \log \left( \frac{1}{1-a_n} \right).
\end{equation*}

Thus, under our assumptions, we obtain an upper bound for the original expression:
\begin{eqnarray*}
&&\frac{1}{2} \left| \log \left( 1+ \frac{ \varepsilon_{(i,n)}}{R^2_{(i,n)}} \right) \right| +\frac{1}{2} \left| \log \left( 1+ \frac{ \varepsilon_{(l,n)}}{R^2_{(l,n)}} \right) \right| \\ &\leq&   \frac{1}{2} \log \left( \frac{1}{1-a_n} \right)+ \frac{1}{2} \log \left( \frac{1}{1-a_n} \right) \\
&=&\log \left( \frac{1}{1-a_n} \right).
\end{eqnarray*}
\end{proof}

\subsection{Asymptotic properties of the separating Hill estimator}\label{sec::main}

\vspace{12pt}

\noindent
Equipped with the lemmas derived in the last section, we will now return to the separating Hill estimator. The bound given in Lemma \ref{Lemma::differences} suffices to control the amount by which replacing the true location and scatter by estimates distorts the value of $\hat{H}_{d}$. The conditions under which the bound is valid are asymptotically nonrestrictive.

\vspace{12pt}

\begin{theorem}\label{thm::consistency}
Let $X_1, X_2, \dots$ be a sequence of i.i.d. random variables with a $d$-dimensional elliptical distribution \eqref{eq::elliptical} that has a positive extreme value index $\gamma$. Assume that $\hat{\mu}_n$ and $\hat{\Sigma}_n$ are consistent estimators of the location and scatter of the distribution, respectively. Now, as $k_n \to \infty,$ $n \to \infty,$ $k_n /n \to 0$, we have that
\begin{equation*}
\hat{H}_{d} \left( S_n, \hat{\mu}_n, \hat{\Sigma}_n , k_n ,n \right) \to_P \gamma,
\end{equation*}
where $\hat{H}_d$ is as in \eqref{eq::sepHill}.
\end{theorem}

\begin{proof}
Let $k_n \to \infty,$ $n \to \infty,$ and let $k_n /n \to 0$.

Let $M_n$ be defined as in Lemma \ref{Lemma::epsilons}. By definition, consistency of $\hat{\mu}_n$ and $\hat{\Sigma}_n$ imply that $M_n \to_P 0$. Since the sequence $k_n \to \infty$ is such that $k_n / n \to 0$, we have that $R_{(k_n+1, n)} \to_P \infty$. Thus the conditions
of Lemma \ref{Lemma::epsilons}
\begin{equation}\label{eq::conditions}
M_n < 1 \quad \text{and} \quad R_{(k_n+1,n)} \geq \frac{M_n}{2 \left( 1- M_n \right) }
\end{equation}
hold with a probability $p_n$ that approaches one.

Define the following auxiliary sequence
\begin{equation*}
b_n = \begin{cases}\begin{array}{ll}
 \log (2), & \text{if } a_n > \frac{1}{2} \\
 \log \left( \frac{1}{1-a_n} \right), & \text{if } a_n \leq \frac{1}{2}
\end{array}
\end{cases}
,
\end{equation*}
where, as in Lemma \ref{Lemma::differences},
\begin{equation*}
a_n = M_n + \frac{M_n}{R_{(k_n+1,n)}} + \frac{M_n}{R^2_{(k_n+1,n)}}.
\end{equation*}
Since $R_{(k_n+1,n)} \to_P \infty$ and $M_n \to_P 0$, we have that $a_n \to_P 0$, and by the continuous mapping theorem $b_n \to_P 0$. Denote by $q_n$ the probability of the event $a_n \leq \frac{1}{2}$.

Assume that the conditions \eqref{eq::conditions} hold, and that $a_n \leq \frac{1}{2}$. Then, by Lemma \ref{Lemma::differences},
\begin{eqnarray*}
&&\left|  \hat{H}_{d} \left( S_n, \hat{\mu}_n, \hat{\Sigma}_n , k_n ,n \right)- \hat{H}_{d} \left( S_n, \mu, \Sigma , k_n ,n \right) \right| \\ &=&\left| \frac{1}{k_n} \sum \limits_{i=1}^{k_n} \log \left( \frac{E_{(i,n)}}{E_{(k_n+1,n)}} \right) - \frac{1}{k_n} \sum \limits_{i=1}^{k_n} \log \left( \frac{R_{(i,n)}}{R_{(k_n+1,n)}} \right) \right| \\
&\leq& \frac{1}{k_n} \sum \limits_{i=1}^{k_n} \left| \log \left( \frac{E_{(i,n)}}{E_{(k_n+1,n)}} \right) - \log \left( \frac{R_{(i,n)}}{R_{(k_n+1,n)}} \right) \right| \leq   \log \left( \frac{1}{1-a_n} \right) \\&=&b_n.
\end{eqnarray*}

Since $p_n, q_n \to 1$, the probability for both conditions \eqref{eq::conditions} and $a_n \leq \frac{1}{2}$ holding simultaneously --- and consequently the probability of the sequence $b_n$ being an upper bound for the difference --- approaches one. The result now follows from $b_n \to_P 0$ and from the fact that the separating Hill estimator is consistent under known location and scatter as was shown in \citet{Ilmonen}.
\end{proof}

\vspace{12pt}

When it comes to the limiting distribution, the rate of convergence of $\hat{\mu}_n$ and $\hat{\Sigma}_n$ also plays a role. It turns out that them being $\sqrt{n}$-consistent is sufficient. As a familiar example, the sample mean vector and the sample covariance matrix are $\sqrt{n}$-consistent if the fourth moments of the marginal distributions are finite.

Let $\mathcal{R}$ be the generating variate of an elliptical distribution with an extreme value index $\gamma > 0$. Define $U$ of $\mathcal{R}$ as in \eqref{eq::U}. Assume that $U \in ERV2_{\gamma, \rho}$ and let $A$ be a corresponding auxiliary function introduced in Definition \ref{def::2ERV}. Below in Theorem \ref{thm::limiting} we simply refer to an elliptical distribution with parameters $\gamma, \rho$ and an auxiliary function $A.$
\vspace{12pt}

\begin{theorem}\label{thm::limiting}
Let $X_1, X_2, \dots$ be a sequence of i.i.d. random variables with a $d$-dimensional elliptical distribution with parameters $\gamma, \rho,$ and an auxiliary function $A$ (see the discussion above). Assume that $\sqrt{n}\left(\mu -\hat{\mu}_n \right)$ and $\sqrt{n}\left(\Sigma -\hat{\Sigma}_n \right)$ converge in distribution. Let $k_n \to \infty,$ $n \to \infty,$ $k_n /n \to 0$, and let
\begin{equation*}
\lambda = \lim \limits_{n \to \infty} A \left( \frac{n}{k_n} \right).
\end{equation*}
Then
\begin{equation*}
\sqrt{k_n} \left(  \hat{H}_{d} \left( S_n, \hat{\mu}_n, \hat{\Sigma}_n , k_n ,n \right) - \gamma \right) \to_D \mathcal{N} \left( \frac{\lambda}{1- \rho}, \gamma^2 \right),
\end{equation*}
where $\hat{H}_d$ is as in \eqref{eq::sepHill}.
\end{theorem}

\begin{proof}
Consider an elliptical distribution with parameters $\gamma, \rho, $ and an auxiliary function $A.$ Assume that $\left(\mu -\hat{\mu}_n \right) = O_p\left(\frac{1}{\sqrt{n}}\right)$ and that $\left(\Sigma -\hat{\Sigma}_n \right) = O_p\left(\frac{1}{\sqrt{n}}\right)$. Let $k_n \to \infty,$ $n \to \infty,$ $k_n /n \to 0$, and let
\begin{equation*}
\lambda = \lim \limits_{n \to \infty} A \left( \frac{n}{k_n} \right).
\end{equation*}

We have that
\begin{eqnarray*}
&&\sqrt{k_n} \left( \hat{H}_{d} \left( S_n, \hat{\mu}_n, \hat{\Sigma}_n , k_n ,n \right) - \gamma \right) \\ &=& \sqrt{k_n} \left( \hat{H}_{d} \left( S_n, \hat{\mu}_n, \hat{\Sigma}_n , k_n ,n \right) - \hat{H}_{d} \left( S_n, \mu_n, \Sigma_n , k_n ,n \right) \right) \\ &&+ \sqrt{k_n} \left( \hat{H}_{d} \left( S_n, \mu_n, \Sigma_n , k_n ,n \right) - \gamma \right).
\end{eqnarray*}
We now consider the first term and show that it converges to zero in probability. The claim then follows from Slutsky's theorem and the limiting normality of the separating Hill estimator under known location and scatter that was proven in \cite{Ilmonen}.

Let $a_n$ and $b_n$ be as in the proof of Theorem \ref{thm::consistency}. Then with a probability that approaches one, we have that

\begin{eqnarray*}
\left| \sqrt{k_n} \left( \hat{H}_{d} \left( S_n, \hat{\mu}_n, \hat{\Sigma}_n , k_n ,n \right) - \hat{H}_{d} \left( S_n, \mu_n, \Sigma_n , k_n ,n \right) \right) \right| \leq  \sqrt{k_n} b_n.
\end{eqnarray*}
It now suffices to show that the right hand side of the inequality above converges to zero in probability.

If $\left| a_n \right| \leq \frac{1}{2}$, then the following holds.
\begin{eqnarray*}
\sqrt{k_n} b_n = \sqrt{k_n} \, \left| \log \left( 1 - a_n \right) \right| &=&  \sqrt{k_n} \, \left|  \sum \limits_{m=1}^{\infty} \frac{\left( -1 \right)^{m+1} }{m} \left( -a_n \right)^m \right|  \\
&\leq& \sqrt{k_n} a_n +  \sqrt{k_n} \sum  \limits_{m=2}^{\infty} \left| \frac{\left( -1 \right)^{m+1} }{m} \left( - a_n \right)^m \right| \\ &\leq&  \sqrt{k_n} a_n + \sqrt{k_n} a_n \sum  \limits_{m=2}^{\infty}  \frac{1}{2^{m-1}m} \\ &=&\left( 1+ S \right) \sqrt{k_n} a_n,
\end{eqnarray*}
where $S$ is the sum of the series
\begin{equation*}
\sum \limits_{m=2}^{\infty}  \frac{1 }{2^{m-1}m} ,
\end{equation*}
that is convergent by the ratio test.

Recall that
\begin{equation*}
\sqrt{k_n}  a_n =  \sqrt{k_n} M_n + \sqrt{k_n} \frac{M_n}{R_{(k_n+1,n)}} + \sqrt{k_n} \frac{M_n}{R^2_{(k_n+1,n)} }.
\end{equation*}
Since $\sqrt{n} \left(  \hat{\Sigma}_n - \Sigma \right)$ is assumed to converge in distribution and
\begin{equation*}
\sqrt{n} \left( \hat{\Sigma}^{-1}_n - \Sigma^{-1} \right) = \hat{\Sigma}^{-1}_n \left(  \sqrt{n} \left( \Sigma -\Sigma_n \right) \right) \Sigma^{-1},
\end{equation*}
we have, by Slutsky's theorem, that $\sqrt{n} \left( \hat{\Sigma}^{-1}_n - \Sigma^{-1} \right)$ also converges in distribution.

The sequence $\sqrt{n} \, \left(  \hat{\mu}_n - \mu \right)$ converges in distribution. Since convergence in distribution implies boundedness in probability, we have that
\begin{eqnarray*}
\sqrt{k_n} \, \left(  \hat{\mu}_n - \mu \right) &=& \sqrt{\frac{k_n}{n}} \sqrt{n} \, \left(  \hat{\mu}_n - \mu \right) \to_P 0 \\ \sqrt{k_n}\left(  \hat{\Sigma}_n - \Sigma \right) &=& \sqrt{\frac{k_n}{n}} \, \sqrt{n} \left(  \hat{\Sigma}_n - \Sigma \right)  \to_P 0,
\end{eqnarray*}
as $k_n / n \to 0$.  In the expression of the sequence $M_n$,  each summand is a product of one of the differences above, and of something that is bounded in probability. Thus, by Slutsky's theorem, $\sqrt{k_n} M_n \to_P 0,$ and it follows that
\begin{equation*}
\left( 1+ S\right) \sqrt{k_n} a_n  \to_P 0.
\end{equation*}
This completes the proof.
\end{proof}

\bibliography{mybib}{}
\bibliographystyle{humannat}

\end{document}